\documentclass[a4paper,twoside]{article}      
\usepackage{amsmath,amssymb,amsfonts,amsthm,amscd,mathtools}
\usepackage{graphics}                 
\usepackage{color}                    
\usepackage{hyperref}                
\usepackage{ mathrsfs }
\usepackage{indentfirst}
\usepackage{float}
\usepackage{bbold}
\usepackage{bm}
\usepackage{slashbox}
\usepackage{enumerate}
\usepackage{url}         
\usepackage{colonequals} 
\usepackage{authblk} 
\usepackage{a4wide}
\usepackage{graphicx}
\hypersetup{
    bookmarks=true,         
    unicode=false,          
    pdftoolbar=true,        
    pdfmenubar=true,        
    pdffitwindow=false,     
    pdfstartview={FitH},    
    pdftitle={My title},    
    pdfauthor={Author},     
    pdfsubject={Subject},   
    pdfcreator={Creator},   
    pdfproducer={Producer}, 
    pdfkeywords={keywords}, 
    pdfnewwindow=true,      
    colorlinks=true,       
    linkcolor=blue,          
    citecolor=blue,        
    filecolor=magenta,      
    urlcolor=cyan           
}

\newtheorem{theorem}{Theorem}[section]
\newtheorem{proposition}[theorem]{Proposition}
\newtheorem{corollary}[theorem]{Corollary}
\newtheorem{lemma}[theorem]{Lemma}
\theoremstyle{definition}
\newtheorem{remark}[theorem]{Remark}

\newtheorem{algorithm}[theorem]{Algorithm}

\def\!{\mathop{\mathrm{!}}}

\def\A{\mathcal{ A}}

\def\B{\mathcal{ B}}

\def\C{\mathcal{ C}}

\def\P{\mathcal{P}}

\def\cov{\mathrm{cov}}
\def\cor{\mathrm{corr}}
\def\var{\mathrm{var}}

\newlength{\boxwidth}
\title{On the expected number of internal equilibria in random evolutionary games  with correlated payoff matrix}
\author[1]{Manh Hong Duong}
\author[2] {Hoang Minh Tran}
\author[3]{The Anh Han}
\affil[1]{Department of Mathematics, Imperial College London, London SW7 2AZ, 
UK.}
\affil[2]{Data Analytics Department, Esmart systems, 1783 Halden, Norway. Email: hoangtm.fami@gmail.com}
\affil[3]{School of Computing, Media and Art, Teesside University, TS1 3BX, UK. Email: T.Han@tees.ac.uk}
\date{\today}
\begin{document}

\maketitle
\begin{abstract}
The analysis of equilibrium points in random  games has been of great interest in evolutionary game theory,  with important implications for understanding of complexity in a dynamical system, such as its behavioural, cultural or biological  diversity. 
The analysis so far has focused on random  games of independent payoff entries. In this paper, we overcome this restrictive assumption by considering multi-player two-strategy evolutionary games where the  payoff matrix entries are correlated random variables. Using techniques from the random polynomial theory we establish  a closed formula for the mean numbers of internal (stable) equilibria. We then characterise the  asymptotic behaviour of this important quantity for large group sizes and study the effect of the correlation. Our results show that decreasing the correlation among payoffs (namely, of a strategist for different group compositions) leads to larger mean numbers of (stable) equilibrium points, suggesting that the system or population behavioural diversity can be promoted by increasing independence of the payoff entries. Numerical results are provided to support the obtained analytical results.
\end{abstract}
\section{Introduction}
\subsection{Motivation}
Evolutionary Game Theory (EGT) was originally introduced in 1973 by Maynard Smith and Price \cite{SP73} as an application of classical game theory to biological contexts,  providing explanations for odd animal behaviours in conflict situations. Since then it has become one of the most diverse and far reaching theories in biology, finding further applications in various fields such as ecology, physics, economics  and computer science \cite{maynard-smith:1982to,axelrod:1984yo,hofbauer:1998mm,nowak:2006bo,broom2013game,Perc2010109,sandholm2010population,HanJaamas2016}. For example, in economics, it has been used to   make predictions in settings where traditional assumptions about agents' rationality and knowledge may not be justified \cite{friedman1998economic,sandholm2010population}. In computer science, EGT has been used extensively to model dynamics and emergent behaviour in multiagent systems \cite{tuyls2007evolutionary,HanBook2013}. Furthermore, EGT has helped explain the 
evolution and emergence of cooperative behaviours in diverse societies, one of the most actively studied and challenging interdisciplinary problems in science \cite{Pennisi93,hofbauer:1998mm,nowak:2006bo,Broom2016}.
 
Similar to the foundational  concept of Nash equilibrium in classical game theory \cite{nash:1950ef}, the study of equilibrium points and their stability in EGT has been of significant importance and extensive research \cite{broom:1997aa, broom:2003aa,gokhale:2010pn,HTG12, gokhale2014evolutionary, DH15, DuongHanJMB2016,Broom2016}. They represent population compositions where all the strategies have the same average fitness, thus  predicting the co-existence of different strategic behaviours or types in a population. 
The major body of such EGT literature has focused on  equilibrium properties in EGT for  concrete games (i.e. games with well-specified payoff structures)  such as the coordination and the public goods games. For example, the maximal number of equilibria, the  stability and attainability of certain equilibrium points, in concrete games have been  well establish; see for example  \cite{broom:1997aa,Broom2000,Pacheco2009, Souza2009,Sasaki2015}.

In contrast to the equilibrium  analysis of concrete games, a recent body of works investigates random games where individual  payoffs obtained from the games are randomly assigned \cite{gokhale:2010pn,HTG12, Galla2013, gokhale2014evolutionary, DH15, DuongHanJMB2016,Broom2016}. 
This analysis has proven useful to provide answers to generic  questions about a dynamical  system such as its overall complexity. Using random games is  useful to model and understand social and biological systems in which  very limited information is available, or where the environment changes so rapidly and frequently that one cannot predict the payoffs of their inhabitants \cite{may2001stability,fudenberg1992evolutionary,HTG12,gross2009generalized}.  Moreover,  even when randomly generated games are not  directly representative for real world scenarios, they are  valuable as a null hypothesis that can be used to sharpen our understanding of what makes real games special \cite{Galla2013}. In general, an important question  posed in these works is that of  
\textbf{what is the expected number, $E(d)$, of internal equilibria in a $d$-player game?} 
An answer to the question provides important insights for  the understanding of the expected levels of   behavioural  diversity or biodiversity one can expect in a dynamical system \cite{levin:2000aa,santos:2008xr,gokhale:2010pn}. It would allow us to predict the level of biodiversity in multiplayer interactions, describing the probability of which a certain state of biodiversity may occur. 
Moreover, computing  $E(d)$  provides useful upper-bounds for the probability $p_m$ that a certain number $m$ of equilibria, is attainted, since  \cite{HTG12}: $p_{m}\leq E(d)/m$. Of particular interest is such an estimate for the probability of attaining the maximal of internal equilibria, i.e. $p_{d-1}$, as in the  Feldman--Karlin conjecture  \cite{Altenberg2010263}.

Mathematically, to find  internal equilibria in a $d$-player  game with two strategies $A$ and $B$, one needs to solve the following polynomial equation for $y>0$ (see Equation  \eqref{eq: eqn for y} and its derivation in Section \ref{sec: replicator}),
\begin{equation}
\label{eq: P1}
P(y):=\sum\limits_{k=0}^{d-1}\beta_k\begin{pmatrix}
d-1\\
k
\end{pmatrix}y^k=0,
\end{equation}
where $\beta_k=a_k-b_k$, with $a_k$ and $b_k$ being  random variables representing  the payoff entries of the game payoff matrix for $A$ and $B$, respectively. Therefore, calculating $E(d)$ amounts to the computation of the expected number of positive zeros of the (random) polynomial $P$. As will be shown in Section \ref{sec: replicator}, the set of positive roots of $P$ is the same as that of the so-called gain function which is a Bernstein polynomial. Thus one can gain information about internal equilibria of a multiplayer game via studying positive roots of Bernstein polynomials. For deterministic multiplayer games, this has already been carried out in the literature \cite{PENA2014}. One of the main goals of this paper is to extend this research to random multiplayer games via studying random polynomials.

In \cite{gokhale:2010pn,HTG12,gokhale2014evolutionary}, the authors provide both numerical and analytical results for games with  a small number of players ($d\leq 4$), focusing on the probability of attaining a maximal number of equilibrium points. These works use a direct approach by solving Equation \eqref{eq: P1},  expressing the positivity of its zeros as domains of conditions for the coefficients and then integrating over these domains to obtain the corresponding probabilities. However, in general, a polynomial of degree five or higher is not analytically solvable \cite{able:1824aa}. Therefore, the direct approach can not be generalised to larger $d$. More recently, in \cite{DH15,DuongHanJMB2016} the authors introduce a novel method using techniques from random polynomials to calculate $E(d)$  with an arbitrary  $d$, under the assumption that the entries of the payoff matrix are \textit{independent} normal random variables. More precisely, they derive a computationally implementable formula for $E(d)$ for arbitrary $d$ and prove the following monotonicity and asymptotic behaviour of $E(d)$:
\begin{equation}
\label{eq: asymp JMB}
\frac{E(d)}{d-1}~\text{is decreasing and}~ \lim_{d\to\infty}\frac{\ln E(d)}{\ln (d-1)}=\frac{1}{2}.
\end{equation}
However, the requirement that the entries of the payoff matrix are independent random variables are rather restricted from both mathematical and biological points of view. In evolutionary game theory, correlations may arise in various scenarios particularly when there are environmental randomness and interaction uncertainty such as in  games of cyclic dominance  \cite{SzolnokiPerc2016},  coevolutionary multigames \cite{SzolnokiPerc2014} or when individual contributions are correlated to the surrounding contexts (e.g. due to limited resource) \cite{santos2012role}, see also recent reviews \cite{Szolnoki2014,Dobramysl2018} for more examples. One might expect some strategies to have many similar properties and hence yield similar results for a given response of the respective opponent \cite{berg:1998aa}. Furthermore, in a multi-player game (such as the public goods games and their generalisations), a strategy's payoffs, which may differ for different group compositions, can be expected to be correlated given a specific nature of the strategy \cite{hardin:1968mm,hauert:2002in, hauert:2006fd, santos:2008xr,pena2012group,Han:2014tl}. Similarly, different strategies' payoffs may be correlated given the same group composition. From a mathematical perspective, the study of real zeros of random polynomials with correlated coefficients has attracted substantial  attention, see e.g. \cite{Sambandham76, BS86, FN05, FN10, FN11}.

In this paper we remove the assumption on the dependence of the coefficients. We will study the expected number of internal equilibria and its various properties for random evolutionary games in which the entries of the payoff matrix are correlated random variables. 

\subsection{Summary of main results}
We now summarise the main results of this paper. More detailed statements will be presented in the sequel sections. We consider $d$-player two-strategy random games in which the coefficients $\beta_k$ ($k\in\{0,\ldots,d-1\}$) can be  correlated random variables,  satisfying that $\mathrm{corr}(\beta_i,\beta_j)=r$ for $i\neq j$ and for some $0\leq r\leq 1$ (see Lemma \ref{lemma:relation_betaAB} about this assumption). 

The main result of the paper is the following theorem which provides a formula for the expected number, $E(r, d)$, of internal equilibria, characterises its asymptotic behaviour and studies the effect of the correlation. 
\begin{theorem}[On the expected number of internal equilibria] \
\begin{enumerate}[1)]
\item (Computational formula for $E(r,d)$)
\begin{equation}
\label{eq: E main}
E(r,d)=2\int_0^1 f(t;r,d)\,dt,
\end{equation}
where the density function $f(t;r,d)$ is given explicitly in \eqref{eq: density of zeros}.
\item (Monotonicity of $E(r,d)$ with respect to $r$) The function $r\mapsto E(r,d)$ decreases for any given $d$.
\item (Asymptotic behaviour of $E(r,d)$ for large $d$) We perform formal asymptotic computations to get
\begin{equation}
\label{eq: asymp}
E(r,d)\begin{cases}\sim \frac{\sqrt{2d-1}}{2}\sim \mathcal{O}(d^{1/2})\quad\text{if}~~ r=0,\\
\sim \frac{d^{1/4}(1-r)^{1/2}}{2\pi^{5/4}r^{1/2}}\frac{8\Gamma\left(\frac{5}{4}\right)^{2}}{\sqrt{\pi}}\sim\mathcal{O}(d^{1/4})\quad\text{if}~~ 0<r<1,\\
=0\quad\text{if}~~ r=1.
\end{cases}
\end{equation}
We compare this asymptotic behaviour numerically with the analytical formula obtained in part 1.
\end{enumerate}
\label{theo: main theo 1}
\end{theorem}
This theorem clearly shows that the correlation $r$ has a significant effect on the expected number of internal equilibria $E(r,d)$. For sufficiently large $d$, when $r$ increases from $0$ (uncorrelated) to $1$ (identical), $E(r,d)$ reduces from $\mathcal{O}(d^{1/2})$ at $r=0$, to $\mathcal{O}(d^{1/4})$ for $0<r<1 $ and to $0$ at $r=1$. This theorem generalises and improves the main results in \cite{DuongHanJMB2016} for the case $r=0$: the asymptotic behaviour, $E(r,d)\sim \frac{\sqrt{2d-1}}{2}$, is stronger than \eqref{eq: asymp JMB}. In addition, as a by-product of our analysis, we provide an asymptotic formula for the expected number of real zeros of a random Bernstein polynomial as conjectured in \cite{Emiris:2010}, see Section \ref{sec: expected zeros Bernstein}.

\subsection{Methodology of the present work}
We develop further the connections between EGT and random/deterministic polynomials theory discovered in \cite{DH15,DuongHanJMB2016}. The integral representation \eqref{eq: E main} is derived from the theory of \cite{EK95}, which provides a general formula for the expected number of real zeros of a random polynomial in a given domain, and the symmetry of the game, see Theorem \ref{theo: integral form}; the monotonicity and asymptotic behaviour of $E(r,d)$ are obtained by using connections to Legendre polynomials, which were described in \cite{DuongHanJMB2016}, see Theorems \ref{theorem:monotonf} and  \ref{theo: asymp}. 

\subsection{Organisation of the paper}
The rest of the paper is organised as follows.  In Section \ref{sec: replicator}, we recall the replicator dynamics for multi-player two-strategy games. In Section \ref{sec: expected}, we prove  and numerically validate the first and the second parts of Theorem \ref{theo: main theo 1}. Section \ref{sec: assymptotic} is devoted to the proof of the last part of Theorem \ref{theo: main theo 1} and its numerical verification. Section \ref{sec: discussion} provides  further discussion and finally, Appendix \ref{sec: App} contains detailed computations and proofs of technical results.

\section{Replicator dynamics}
\label{sec: replicator}
A fundamental model of evolutionary game theory  is the replicator dynamics  \cite{taylor:1978wv,zeeman:1980ze,hofbauer:1998mm,schuster:1983le,nowak:2006bo}, describing that whenever a strategy has a fitness larger than the  average fitness of the population, it is expected to  spread. From the replicator dynamics one then can derive a polynomial equation that an internal equilibria of a multiplayer game satisfies . To this end, we consider an infinitely large population with two strategies, A and B. Let  $x$, $0 \leq x \leq 1$, be the frequency of strategy A. The frequency of strategy B is thus $(1-x)$. The interaction of the individuals in the population is in randomly selected groups of $d$ participants, that is, they play and obtain their fitness from   $d$-player games. The game is defined through a $(d-1)$-dimensional payoff matrix \cite{gokhale:2010pn}, as follows. Let $a_k$  (resp., $b_k$) be the payoff of an A-strategist (resp., B) in a group  containing  $k$ A strategists (i.e. $d-k$ B strategists). In this paper, we consider symmetric games where the payoffs do not depend on the ordering of the players. Asymmetric games will be studied in our forthcoming paper~\cite{DuongTranHan2017TMP}. In the symmetric case, the average payoffs of $A$ and $B$  are, respectively 
\begin{equation*}
\pi_A= \sum\limits_{k=0}^{d-1}a_k\begin{pmatrix}
d-1\\
k
\end{pmatrix}x^k (1-x)^{d-1-k}\quad\text{and}\quad
\pi_B = \sum\limits_{k=0}^{d-1}b_k\begin{pmatrix}
d-1\\
k
\end{pmatrix}x^k (1-x)^{d-1-k}.
\end{equation*}
Internal equilibria are those points that satisfy the condition that the  fitnesses of both strategies are the same $\pi_A=\pi_B$, which gives rise to $g(x)=0$ where $g(x)$ is the so-called gain function given by \cite{BACH2006,PENA2014}
\begin{equation*}
g(x)=\sum\limits_{k=0}^{d-1}\beta_k \begin{pmatrix}
d-1\\
k
\end{pmatrix}x^k (1-x)^{d-1-k},
\end{equation*}
where $\beta_k = a_k - b_k$. Note that this equation can also be derived from the definition of an evolutionary stable strategy (ESS), see e.g., \cite{broom:1997aa}. As also discussed in that paper, the evolutionary solution of the game (such as the set of ESSs or the set of stable rest points of the replicator dynamics) involves not only finding the roots of the gain function $g(x)$ but
also determining the behaviour of $g(x)$ in the vicinity of such roots. We also refer the reader to \cite{taylor:1978wv,zeeman:1980ze} and references therein for further discussion on relations between ESSs and game dynamics. Using the transformation $y= \frac{x}{1-x}$, with $0< y < +\infty$, and dividing $g(x)$ by $(1-x)^{d-1}$ we obtain the following polynomial equation for $y$
\begin{equation}
\label{eq: eqn for y}
P(y):=\sum\limits_{k=0}^{d-1}\beta_k\begin{pmatrix}
d-1\\
k
\end{pmatrix}y^k=0.
\end{equation}
As in \cite{gokhale:2010pn, DH15, DuongHanJMB2016}, we are interested in random games where $a_k$ and $b_k$ (thus $\beta_k$), for $0\leq k\leq d-1 $, are random variables. However, in contrast to these papers where  $\beta_k$ are assumed to be independent, we  analyse here a more general case where they are correlated. In particular, we consider that any pair $\beta_i$ and $\beta_j$, with $0\leq i\neq j \leq d-1$, have a correlation $r$ ($0\leq r \leq1$). In general, $r = 0$ means $\beta_i$ and $\beta_j$ are independent while when $r = 1$ they have a (perfectly) linear correlation, and the larger $r$ is the stronger they are correlated. It is noteworthy  that this type of dependency between the coefficients is common in the literature on evolutionary game theory \cite{berg:1998aa,Galla2013} as well as  random polynomial theory \cite{Sambandham76,BS86, FN11}. 

The next lemma shows how this assumption arises naturally  from simple assumptions on the game payoff entries. To state the lemma,  let  $\cov(X,Y)$ and $\cor(X,Y)$ denote the covariance and correlation between  random variables $X$ and $Y$, respectively;  moreover,  $\var(X)=\cov(X,X)$ denotes the variance of $X$.
\begin{lemma}
\label{lemma:relation_betaAB}
Suppose that, for $0\leq i\neq j\leq d-1$,
\begin{itemize}
\item $\var(a_i)=\var(b_i)=\eta^2$,
\item $\cor(a_i,a_j)=r_a$,  $\cor(b_i,b_j)=r_b$,
\item $\cor(a_i,b_j)=r_{ab}$, $\cor(a_i,b_i)=r'_{ab}$.
\end{itemize}
Then, the correlation between $\beta_i$ and $\beta_j$, for $1\leq i\neq j\leq d-1$, is given by
\begin{equation}
\cor(\beta_i,\beta_j)=\frac{r_a+r_b-2r_{ab}}{2(1-r'_{ab})}, 
\end{equation}
which is a constant. Clearly, it  increases with $r_a$,  $r_b$ and $r'_{ab}$ while decreasing with $r_{ab}$. Moreover, if $r_a+r_b=2r_{ab}$ then $\beta_i$ and $\beta_j$ are independent. Also, if $r_{ab} = r'_{ab} = 0$, i.e. when payoffs from different strategists are independent, we have:    $\cor(\beta_i,\beta_j)=\frac{r_a+r_b}{2}$. If we further assume that $r_a = r_b = r$, then $\cor(\beta_i,\beta_j)= r$.  
\end{lemma}
\begin{proof} See Appendix \ref{sec: proof lemma}.
\end{proof}
The assumptions in Lemma \ref{lemma:relation_betaAB} mean that a strategist's payoffs for different group compositions have a constant correlation, which  in general  is different from the cross-correlation of payoffs for  different strategists. These  assumptions arise naturally for example in a multi-player game (such as the public goods games and their generalisations), since a strategist's payoffs,  which may differ for different group compositions, can be expected to be correlated given a specific nature of the strategy (e.g. cooperative vs. defective strategies in the public goods games). 
 These natural assumptions regarding payoffs' correlations are just to ensure  the pairs $\beta_i$ and $\beta_j$, $0\leq i\neq j\leq d-1$, have a constant correlation. Characterising  the general case where $\beta_i$ and $\beta_j$ have varying correlations would be mathematically interesting but is out of the scope of this paper. We will  discuss further  this issue particularly for  other   types of correlations in Section \ref{sec: discussion}.

\section{The expected number of internal equilibria $E(r,d)$}
\label{sec: expected}
We consider the case where $\beta_k$ are standard normal random variables but assume that all the pairs $\beta_i$ and $\beta_j$, for $0\leq i\neq j\leq d-1$, have the same correlation $0 \leq r \leq 1$ (cf. Lemma \ref{lemma:relation_betaAB}).

In this section, we study the expected number of internal equilibria $E(r,d)$. The starting point of the analysis of this section is an improper integral to compute $E(r,d)$ as a direct application of the Edelman-Kostlan theorem \cite{EK95}, see Lemma \ref{lemma: density of zeros}. We then further simplify this formula to obtain a more computationally  tractable one (see Theorem \ref{theo: integral form}) and then prove a monotone property of $E(r,d)$ as a function of the correlation $r$, see Theorem \ref{theorem:monotonf}.
\subsection{Computations of $E(r,d)$}
\begin{lemma}
\label{lemma: density of zeros}
Assume that $\beta_k$ are standard normal random variables and that for any $i\neq j$, the correlation between $\beta_i$ and $\beta_{j}$ is equal to $r$ for some $0 \leq r \leq 1$. Then the expected  number of internal equilibria, $E(r,d)$, in a $d$-player random game with two strategies is given by
\begin{equation}
\label{eq: E(d)}
E(r,d)=\int_0^\infty f(t; r,d)\,dt,
\end{equation}
where
\begin{multline}
\label{eq: density of zeros}
[\pi\,f(t; r,d)]^2=\frac{(1-r)\sum\limits_{i=0}^{d-1}i^2\begin{pmatrix}
d-1\\
i
\end{pmatrix}^2t^{2(i-1)}+r (d-1)^2(1+t)^{2(d-2)}}{(1-r)\sum\limits_{i=0}^{d-1}\begin{pmatrix}
d-1\\
i
\end{pmatrix}^2t^{2i}+r(1+t)^{2(d-1)}}
\\-\left[\frac{(1-r)\sum\limits_{i=0}^{d-1}i\begin{pmatrix}
d-1\\
i
\end{pmatrix}^2t^{2i-1}+r(d-1)(1+t)^{2d-3}}{(1-r)\sum\limits_{i=0}^{d-1}\begin{pmatrix}
d-1\\
i
\end{pmatrix}^2t^{2i}+r(1+t)^{2(d-1)}}\right]^2.
\end{multline}
\end{lemma}
\begin{proof}
According to \cite{EK95} (see also \cite{DH15, DuongHanJMB2016}), we have
\begin{equation*}
E(r,d)=\int_0^\infty f(t;r,d)\,dt,
\end{equation*}
where the density function $f(t;r,d)$ is determined by
\begin{equation}
\label{eq: f}
f(t;r,d)=\frac{1}{\pi}\left[\frac{\partial^2}{\partial x\partial y}\Big(\log v(x)^T\C v(y)\Big)\Big\vert_{y=x=t}\right]^\frac{1}{2}, 
\end{equation}
with the covariance matrix $\C$ and the vector $v$ are given by
\begin{equation}
\label{eq: C and v}
\C_{ij}=\begin{cases}
\begin{pmatrix}
d-1\\
i
\end{pmatrix}^2,\quad\text{if}~~ i=j\\
r\begin{pmatrix}
d-1\\
i
\end{pmatrix}\begin{pmatrix}
d-1\\
j
\end{pmatrix},\quad\text{if}~~i\neq j.
\end{cases} \quad \text{and} \quad v(x)=\begin{pmatrix}
1\\
x\\
\vdots\\
x^{d-1}
\end{pmatrix}.
\end{equation}
Let us define
\begin{align}
H(x,y)&:=v(x)^T \C v(y)\nonumber
\\&=\sum\limits_{i=0}^{d-1}\begin{pmatrix}
d-1\\
i
\end{pmatrix}^2x^iy^i+r\sum_{i\neq j=0}^{d-1}\begin{pmatrix}
d-1\\
i
\end{pmatrix}\begin{pmatrix}
d-1\\
j
\end{pmatrix}x^iy^j\nonumber
\\&=(1-r)\sum\limits_{i=0}^{d-1}\begin{pmatrix}
d-1\\
i
\end{pmatrix}^2x^iy^i+r\left(\sum_{i=0}^{d-1}\begin{pmatrix}
d-1\\
i
\end{pmatrix}x^i\right)\left(\sum_{j=0}^{d-1}\begin{pmatrix}
d-1\\
j
\end{pmatrix}y^j\right).
\label{eq: H}
\end{align}
Then we compute
\[
\frac{\partial^2}{\partial_x\partial_y}(\log v(x)^T\C v(y))=\frac{\partial^2}{\partial_x\partial_y} \log H(x,y)=\frac{\partial^2_{xy}H(x,y)}{H(x,y)}-\frac{\partial_x H(x,y)\partial_y H(x,y)}{H(x,y)^2}.
\]
Particularly, for $y=x=t$, we obtain
\begin{align*}
\frac{\partial^2}{\partial_x\partial_y}(\log v(x)^T\C v(y))\Big\vert_{y=x=t}&=\left(\frac{\partial^2_{xy}H(x,y)}{H(x,y)}-\frac{\partial_x H(x,y)\partial_y H(x,y)}{H(x,y)^2}\right)\Big\vert_{y=x=t}
\\&=\frac{\partial^2_{xy}H(x,y)\big\vert_{y=x=t}}{H(t,t)}-\left(\frac{\partial_x H(x,y)\big\vert_{y=x=t}}{H(t,t)}\right)^2.
\end{align*} 
Using \eqref{eq: H} we can compute each term on the right hand side of the above expression explicitly
\begin{subequations}
\label{eq: temp1}
\begin{align}
&H(t,t)=(1-r)\sum\limits_{i=0}^{d-1}\begin{pmatrix}
d-1\\
i
\end{pmatrix}^2t^{2i}+r\left(\sum\limits_{i=0}^{d-1}\begin{pmatrix}
d-1\\
i
\end{pmatrix}t^i\right)^2,
\\&\partial_x H(x,y)\big\vert_{y=x=t}=(1-r)\sum\limits_{i=0}^{d-1}i\begin{pmatrix}
d-1\\
i
\end{pmatrix}^2t^{2i-1}+r\left(\sum_{i=0}^{d-1}i\begin{pmatrix}
d-1\\
i
\end{pmatrix}t^{i}\right)\left(\sum_{j=0}^{d-1}\begin{pmatrix}
d-1\\
j
\end{pmatrix}t^{j-1}\right),
\\& \partial^2_{xy}H(x,y)\big\vert_{y=x=t}=(1-r)\sum\limits_{i=0}^{d-1}i^2\begin{pmatrix}
d-1\\
i
\end{pmatrix}^2t^{2(i-1)}+r\left(\sum\limits_{i=0}^{d-1}i\begin{pmatrix}
d-1\\
i
\end{pmatrix}t^{i-1}\right)^2.
\end{align}
\end{subequations}
We can simplify further the above expressions using the following computations which are attained from the binomial theorem and its derivatives
\begin{subequations}
\label{eq: temp2}
\begin{align}
&\left(\sum_{i=0}^{d-1}\begin{pmatrix}
d-1\\
i
\end{pmatrix}t^i\right)^2=(1+t)^{2(d-1)},
\\&\left(\sum\limits_{i=0}^{d-1}i\begin{pmatrix}
d-1\\
i
\end{pmatrix}t^{i-1}\right)^2=\left(\frac{d}{dt}\sum_{i=0}^{d-1}\begin{pmatrix}
d-1\\
i
\end{pmatrix}t^i\right)^2=\left(\frac{d}{dt}(1+t)^{d-1}\right)^2=(d-1)^2 (1+t)^{2(d-2)},
\\& \left(\sum_{i=0}^{d-1}i\begin{pmatrix}
d-1\\
i
\end{pmatrix}t^{i}\right)\left(\sum_{j=0}^{d-1}\begin{pmatrix}
d-1\\
j
\end{pmatrix}t^{j-1}\right)=\frac{1}{2}\frac{d}{dt}\left(\sum_{i=0}^{d-1}\begin{pmatrix}
d-1\\
i
\end{pmatrix}t^i\right)^2=\frac{1}{2}\frac{d}{dt}(1+t)^{2(d-1)}\notag
\\&\hspace*{6.17cm}=(d-1)(1+t)^{2d-3}.
\end{align}
\end{subequations}
Substituting \eqref{eq: temp1} and \eqref{eq: temp2} back into \eqref{eq: f},  we obtain \eqref{eq: density of zeros} and complete the proof.
\end{proof}
Next we will show that, as in the case $r=0$ studied in \cite{DH15,DuongHanJMB2016}, the improper integral \eqref{eq: E(d)} can be reduced to a definite integral from $0$ to $1$. A crucial property enables us to do so is the symmetry of the strategies. The main result of this section is the following theorem (cf. Theorem \ref{theo: main theo 1}--(1)).
\begin{theorem}
\label{theo: integral form}
 \begin{enumerate}[(1)]
\item The density function $f(t;r,d)$ satisfies that
\begin{equation}
\label{eq: f(1/t) vs f(t)}
f(1/t; r,d)=t^2 f(t;r,d).
\end{equation}
\item (Computable formula for $E(r,d)$). $E(r,d)$ can be computed via
\begin{equation}
\label{eq: finite integral}
E(r,d)=2\int_0^1\, f(t)dt=2\int_1^\infty f(t)\,dt.
\end{equation}
\end{enumerate}
\end{theorem}
\begin{proof}
The proof of the first part is lengthy and is given in Appendix \ref{sec: symmetry of game}. Now we prove the second part. We have
\begin{equation}
\label{eq: E1}
E(r,d)=\int_0^\infty f(t;r,d)\,dt=\int_0^1 f(t;r,d)\,dt+\int_1^\infty f(t;r,d)\,dt.
\end{equation}
By changing of variables $t:=\frac{1}{s}$, the first integral on the right-hand side of \eqref{eq: E1} can be transformed as
\begin{equation}
\label{eq: E2}
\int_0^1 f(t;r,d)\,dt=\int_{1}^\infty f(1/s;r,d)\frac{1}{s^2}\,ds=\int_1^\infty f(s;r,d)\,ds,
\end{equation}
where we have used \eqref{eq: f(1/t) vs f(t)} to obtain the last equality. The assertion \eqref{eq: finite integral} is then followed from \eqref{eq: E1} and \eqref{eq: E2}.
\end{proof}
As in \cite{DuongHanJMB2016}, we can interpret the first part of Theorem \ref{theo: integral form} as a symmetric property of the game. We recall that $t=\frac{y}{1-y}$, where $y$ and $1-y$ are respectively the fractions of strategy 1 and 2. We write the density function $f(t;r,d)$ in terms of $y$ using the change of variable formula as follows.
\begin{equation*}
f(t;r,d)\,dt=f\Big(\frac{y}{1-y};r,d\Big)\frac{1}{(1-y)^2}\,dy:=g(y;r,d)\,dy,
\end{equation*}
where 
\begin{equation}
g(y;r,d):=f\Big(\frac{y}{1-y};r,d\Big)\frac{1}{(1-y)^2}.
\end{equation}
The following lemma expresses the symmetry of the strategies (swapping the index labels converts an equilibrium at $y$ to one at $1-y$).
\begin{corollary}
\label{lemma:symmetry of g}
The function $y\mapsto g(y;r,d)$ is symmetric about the line $y=\frac{1}{2}$, i.e.,
\begin{equation}
\label{eq: g}
g(y;r,d)=g(1-y;r,d).
\end{equation}
\end{corollary}
\begin{proof}
The equality \eqref{eq: g} is a direct consequence of \eqref{eq: f(1/t) vs f(t)}. We have
\begin{align*}
g(1-y;r,d)=f\Big(\frac{1-y}{y};r,d\Big)\frac{1}{y^2}\overset{\eqref{eq: f(1/t) vs f(t)}}{=}f\Big(\frac{y}{1-y};r,d\Big)\frac{y^2}{(1-y)^2}\frac{1}{y^2}=f\Big(\frac{y}{1-y};r,d\Big)\frac{1}{(1-y)^2}=g(y;r,d).
\end{align*}
\end{proof}
\subsection{Monotonicity of $r\mapsto E(r,d)$} 
In this section we study the monotone property of $E(r,d)$ as a function of the correlation $r$. 
The main result of this section is the following theorem on the monotonicity of $r\mapsto E(r,d)$ (cf. Theorem \ref{theo: main theo 1}--(2)).
\begin{theorem} 
\label{theorem:monotonf}
The function $r\mapsto f(t;r,d)$ is decreasing. As a consequence, $r\mapsto E(r,d)$ is also decreasing.
\end{theorem}
\begin{proof}
We define the following notations:
\begin{align*}
&M_1=M_1(t;r,d)=\sum_{i=0}^{d-1}\begin{pmatrix}
d-1\\i
\end{pmatrix}^2t^{2i}, \quad M_2=M_2(t;r,d)=(1+t)^{2(d-1)},
\\& A_1=A_1(t;r,d)=\sum\limits_{i=0}^{d-1}i^2\begin{pmatrix}
d-1\\
i
\end{pmatrix}^2t^{2(i-1)}, \quad A_2=A_2(t;r,d)=(d-1)^2(1+t)^{2(d-2)},
\\& B_1=B_1(t;r,d)=\sum\limits_{i=0}^{d-1}i\begin{pmatrix}
d-1\\
i
\end{pmatrix}^2t^{2i-1}, \quad B_2=B_2(t;r,d)=(d-1)(1+t)^{2d-3},
\\& M=(1-r)M_1+ r M_2, \quad A=(1-r) A_1+r A_2, \quad B=(1-r)B_1+ r B_2.
\end{align*}
Then the density function $f(t;r,d)$ in \eqref{eq: density of zeros} can be 
written as
\begin{equation}
\label{eq: relation A-B-M}
(\pi f(t;r,d))^2=\frac{A M-B^2}{M^2}.
\end{equation}
Taking the derivation with respect to $r$ of the right hand side of \eqref{eq: relation A-B-M} we obtain 
\begin{align*}
\label{eq: relation A-B-M}
\frac{\partial}{\partial r}\left(\frac{A M-B^2}{M^2}\right)&=\frac{(A^\prime M + M^\prime A -2 B B^\prime) M^2 - 2 (A M  -2 B^2 ) M M^\prime }{M^4}
\\& =\frac{(A^\prime M + M^\prime A -2 B B^\prime) M - 2 (A M  - B^2 ) M^\prime }{M^3}
\\ &=\frac{2 B (B M^\prime -  B^\prime M) - M (A M^\prime  -M A^\prime )  }{M^3}
\\ &\overset{(*)}{=}\frac{2 B (B_1 M_2 - M_1 B_2 ) - M (A_1 M_2 - M_1 A_2 )  }{M^3}
\\ &=\frac{2 B \left(B_1 (1+t)^{2(d-1)} - M_1 (d-1)(1+t)^{2d-3} \right) - M \left(A_1 (1+t)^{2(d-1)} - M_1 (d-1)^2(1+t)^{2(d-2)} \right)  }{M^3}
\\&=\frac{ (1+t)^{2d-4}\left\{2(t+1)B \left[B_1(1+t) - M_1 (d-1) \right] - M \left[A_1 (1+t)^2 - M_1 (d-1)^2\right]\right\}}{M^3}.
\end{align*}
Note that to obtain (*) above we have used the following simplifications 
\begin{align*}
B M^\prime -  B^\prime M &= \left[B_1 + r (B_2 - B_1)\right] (M_2 - M_1) - (B_2 - B_1) \left[M_1 + r (M_2 - M_1)\right] 
\\ &=  B_1 (M_2-M_1) - (B_2 - B_1) M_1 
\\ &=  B_1 M_2 - M_1 B_2, 
\end{align*}
and similarly, 
$$A M^\prime -  A^\prime M = A_1 M_2 - M_1 A_2. $$
Since $M>0$ and according to Proposition \ref{lemma:support_deriv_theorem1}, 
\[
2(t+1)B \Big[B_1(1+t) - M_1 (d-1) \Big] - M \Big[A_1 (1+t)^2 - M_1 (d-1)^2\Big]\leq 0,
\] 
it follows that
\[
\frac{\partial}{\partial r}\left(\frac{A M-B^2}{M^2}\right)\leq 0.
\]
The assertion of the theorem is then followed from this and \eqref{eq: relation A-B-M}.
\end{proof}

As a consequence, we can derive the monotonicity property of the number of stable equilibrium points, denoted by $SE(r,d)$. It is based on the following property of stable equilibria in multi-player two-strategy evolutionary games, which has been proved in \cite[Theorem 3]{HTG12} for  payoff matrices  with independent entries. We provide a similar proof below for matrices with \textit{exchangeable} payoff entries. We need the following auxiliary lemma whose proof is presented in  Appendix~\ref{sec: symmetry of betas}.
\begin{lemma}
\label{lem: symmetry of betas} Let $X$ and $Y$ be two exchangeable random variables, i.e. their joint probability distribution $f_{X,Y}(x,y)$ is symmetric, $f_{X,Y}(x,y)=f_{X,Y}(y,x)$. Then $Z=X-Y$ is symmetrically distributed about $0$, i.e., its probability distribution satisfies $f_Z(z)=f_Z(-z)$. In addition, if $X$ and $Y$ are iid then they are exchangeable.
\end{lemma}
\begin{theorem} 
\label{theorem:SE}
Suppose that $a_k$ and $\beta_k$ are exchangeable random variables. For  $d$-player evolutionary games with two strategies,  the following holds
\begin{equation} 
SE(r,d) = \frac{1}{2} E(r,d).
\end{equation}
\end{theorem}

\begin{proof}
The replicator equation in this game is  given by \cite{hauert:2006fd,gokhale:2010pn}
\begin{equation}
\dot{x} = x(1-x)\sum_{k = 0}^{d-1}  \beta_k  \ \tbinom {d-1} k \ x^k (1-x)^{d-1-k}.
\end{equation} 
Suppose  $ x^\ast \in (0,1)$ is an internal equilibrium of the system and $h(x)$ be the polynomial on the right hand side of the equation.  Since  
$ x^\ast$ is stable if and only if
$h^\prime( x^\ast) < 0$ which can be simplified to \cite{HTG12}
\begin{equation}
\label{eq:stability_2strats_eq}
\sum_{k = 1}^{d-1} k \beta_k  \ \tbinom {d-1} k{y^\ast}^{k-1}  < 0,
\end{equation}  
where $y^\ast = \frac{x^\ast}{1-x^\ast}$.
As a system admits the same set of equilibria if we change the sign of all $\beta_k$ simultaneously, and for such a change the above inequality would change the direction (thus the stable equilibrium $x^\ast$ would become unstable), all we need to show for the theorem to hold is that $\beta_k$ has a symmetric density function.  This is guaranteed by Lemma~\ref{lem: symmetry of betas} since $\beta_k=a_k-b_k$ where $a_k$ and $b_k$ are exchangeable.
\end{proof} 
\begin{corollary}
\label{lemma:monotonSE}
Under the assumption of Theorem \ref{theorem:SE}, the expected number of stable equilibrium points SE(r,d) is a decreasing function with respect to $r$. 
\end{corollary}
\begin{proof}
This is a direct consequence of Theorems \ref{theorem:monotonf} and  \ref{theorem:SE}.  \end{proof}
\subsection{Monotonicity of $E(r,d)$: numerical investigation}
\label{sec: numerics qual}
In this section, we numerically validate the analytical results obtained in the previous section. In Figure \ref{fig:piechart}, we plot the functions $r\mapsto E(r,d)$ for several values of $d$ (left panel) and $d\mapsto E(r,d)$ for different values of $r$ using formula \ref{eq: E(d)} (right panel). In the panel  on the left we also show the value of $E(r,d)$ obtained from samplings. That is,  we generate $10^6$ samples of $\beta_k (0\leq k\leq d-1)$ where $\beta_k$ are normally distributed random variables satisfying that $\cor(\beta_i,\beta_j)=r$ for $0\leq i\neq j\leq d-1$. For each  sample  we solve Equation \eqref{eq: eqn for y} to obtain the corresponding  number internal equilibria (i.e. the number of positive zeros of the polynomial equation). By averaging over all the $10^6$  samples we obtain the probability of observing $m$ internal equilibria,  $\bar{p}_m$, for each $0\leq m\leq d-1$. Finally the mean or expected  number of internal equilibria is calculated as  $E(r,d)=\sum_{m=0}^{d-1}m\cdot \bar{p}_m$. The figure shows the agreement of results obtained from analytical and sampling methods. In addition, it also demonstrates the decreasing property of $r\mapsto E(r,d)$, which was  proved in Theorem \ref{theorem:monotonf}. Additionally, we observe that $E(r,d)$   increases with the group size, $d$.

Note that to generate correlated normal random variables, we use the following algorithm that can be found in many textbooks, for instance \cite[Section 4.1.8]{nowak2000}.
\begin{algorithm} Generate $n$ correlated Gaussian distributed random variables $\mathbf{Y}=(Y_1,\ldots, Y_n)$, $\mathbf{Y}\sim \mathcal{N}(\mu,\Sigma)$, given the mean vector $\mu$ and the covariance matrix $\Sigma$.
\begin{enumerate}[Step 1.]
\item Generate a vector of uncorrelated Gaussian random variables, $\mathbf{Z}$,
\item Define $\mathbf{Y}=\mu+ C \mathbf{Z}$ where $C$ is the square root of $\Sigma$ (i.e., $C C^T=\Sigma$).
\end{enumerate}
\label{al: algorithm}
\end{algorithm}
The square root of a matrix can be found using the Cholesky decomposition. These two steps are easily implemented in Mathematica. 
\begin{figure}[H]
\centering
\includegraphics[width = \linewidth]{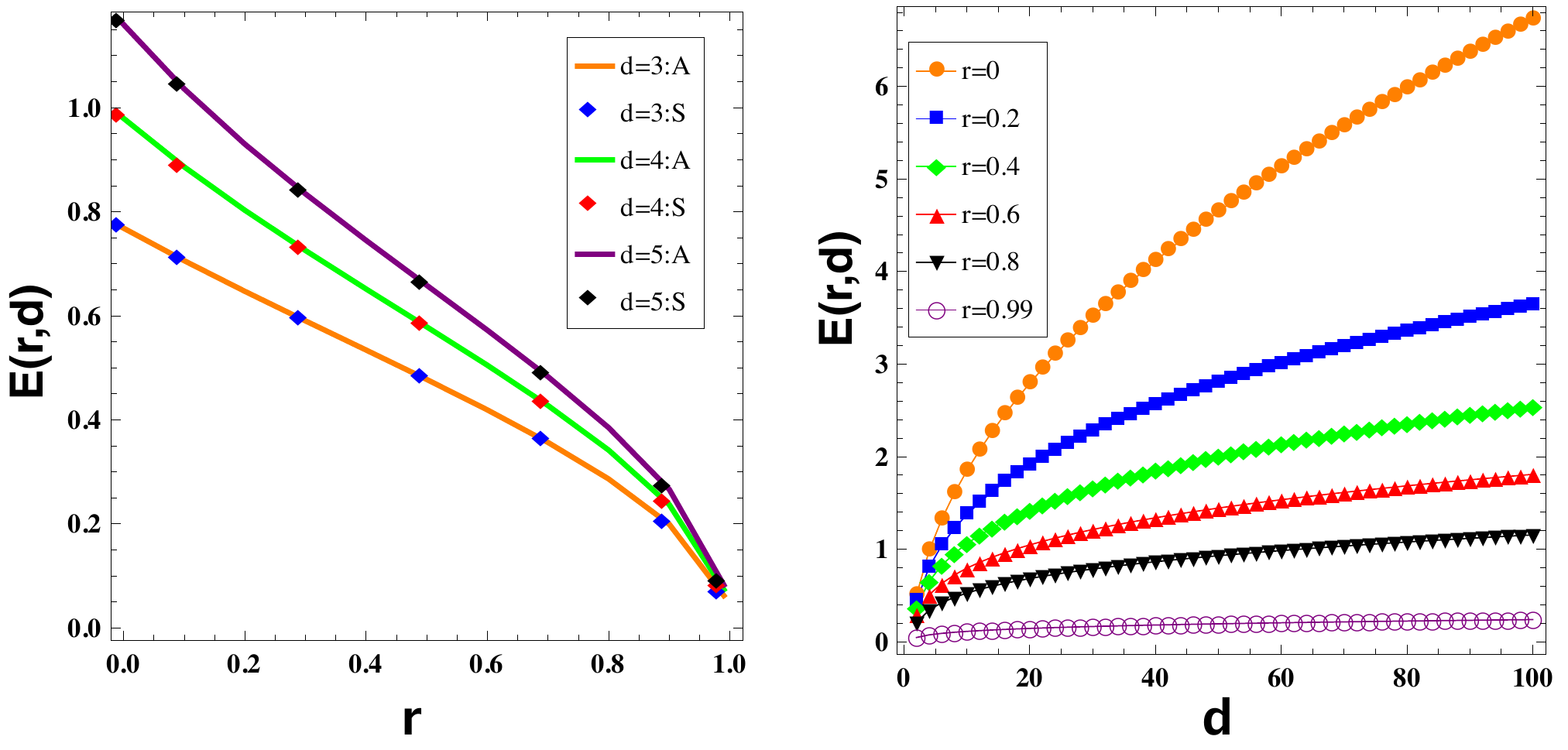}
\caption{\textbf{(Left)} Plot of $r\mapsto E(r,d)$ for different values of $d$.  The solid lines are generated from analytical (\textbf{A}) formulas of $E(r,d)$ as defined in  Equation \eqref{eq: E(d)}.  The solid diamonds capture simulation (\textbf{S}) results obtained by  averaging  over $10^6$ samples of  $\beta_k$ ($1 \leq k \leq d-1$), where these $\beta_k$ are correlated, normally standard random variables. To generate correlated random variables, the algorithm described in Algorithm \ref{al: algorithm} was used.   \textbf{(Right)} Plot of $d\mapsto E(r,d)$ for different values of $r$. We observe that $E(r,d)$ decreases with respect to $r$ but increases with respect to $d$.}
\label{fig:piechart}
\end{figure}

\section{Asymptotic behaviour of $E(r,d)$}
\label{sec: assymptotic}
\subsection{Asymptotic behaviour of $E(r,d)$: formal analytical computations}
In this section we perform formal asymptotic analysis to understand the behaviour of $E(r,d)$ when $d$ becomes large. 
\begin{proposition} 
\label{theo: asymp}
We have the following asymptotic behaviour of $E(r,d)$ as $d\to\infty$
\begin{equation*}
E(r,d)\begin{cases}\sim \frac{\sqrt{2d-1}}{2}\quad\text{if}~~ r=0,\\
\sim \frac{d^{1/4}(1-r)^{1/2}}{2\pi^{5/4}r^{1/2}}\frac{8\Gamma\left(\frac{5}{4}\right)^{2}}{\sqrt{\pi}}\quad\text{if}~~ 0<r<1,\\
=0\quad\text{if}~~ r=1.
\end{cases}
\end{equation*}
\end{proposition}
\begin{proof}
We consider the case $r=1$ first. In this case, we have
\[
M(t)=M_2(t)=(1+t)^{2(d-1)}, A(t)=A_2(t)=(d-1)^2(1+t)^{2(d-2)}, B(t)=B_2(t)=(d-1)(1+t)^{2d-3}.
\]
Since $A_2(t)M_2(t)-B_2^2(t)=0$, we obtain $f(t;1,d)=0$. Therefore $E(1,d)=0$.
\\ \ \\
We now deal with the case $0\leq r<1$. According to \cite[Example 2, page 229]{BO99}, \cite{Wang12}, for any $x>1$
\begin{equation*}
P_d(x)=\frac{1}{\sqrt{2d\pi}}\frac{(x+\sqrt{x^2-1})^{d+1/2}}{(x^2-1)^{1/4}}+\mathcal{O}(d^{-1})\quad\text{as}~~ d\to\infty.
\end{equation*}
Therefore,
\begin{align*}
&M_1=(1-t^2)^{d-1}P_{d-1}\left(\frac{1+t^2}{1-t^2}\right)\sim \frac{1}{\sqrt{4\pi(d-1)t}}(1+t)^{2d-1}\quad\text{and}
\\& M\sim(1-r)\frac{1}{\sqrt{4\pi(d-1)t}}(1+t)^{2d-1}+r(1+t)^{2d-2}.
\end{align*}
Using the relations between $A_1, B_1$ and $M_1$ in \eqref{eq: abm1}, we obtain
\begin{align*}
&A\sim (d-1)^{2}r(t+1)^{2(d-2)}+\frac{(2d-1)(t+1)^{2d-2}}{8t\sqrt{\pi}\sqrt{(d-1)t}}-\frac{(d-1)(t+1)^{2d-1}}{16t\sqrt{\pi}((d-1)t)^{3/2}}
\\&\qquad\qquad +\frac{1}{4}\left(\frac{(2d-2)(2d-1)(t+1)^{2d-3}}{2\sqrt{\pi}\sqrt{(d-1)t}}-\frac{(d-1)(2d-1)(t+1)^{2d-2}}{2\sqrt{\pi}((d-1)t)^{3/2}}+\frac{3(d-1)^{2}(t+1)^{2d-1}}{8\sqrt{\pi}((d-1)t)^{5/2}}\right)
\\& B\sim (d-1)r(t+1)^{2d-3}+\frac{1}{2}(1-r)\left(\frac{(2d-1)(t+1)^{2d-2}}{2\sqrt{\pi}\sqrt{(d-1)t}}-\frac{(d-1)(t+1)^{2d-1}}{4\sqrt{\pi}((d-1)t)^{3/2}}\right).
\end{align*}
Therefore, we get
\begin{align*}
f^2=\frac{1}{\pi^2}\frac{AM-B^2}{M^2}\sim\frac{(1-r)\left(2(1-2d)(r-1)t(t+1)+\sqrt{\pi}r(t(8d+t-6)+1)\sqrt{(d-1)t}\right)}{8\pi^{2}t^{2}(t+1)\left((r-1)(t+1)-2\sqrt{\pi}r\sqrt{(d-1)t}\right)^{2}}.
\end{align*}
Denote the expression on the right-hand side by $f_a^2$. If $r=0$, we have
\[
f_a^2=\frac{2(2d-1)t(t+1)}{8\pi^2t^2(t+1)(t+1)^2}=\frac{2d-1}{4\pi^2 t(t+1)^2},
\]
which means
\[
f_a=\frac{\sqrt{2d-1}}{2\pi \sqrt{t} (t+1)}.
\]
Therefore
\begin{equation*}
E\sim E_a:=2\int_0^1 f_a \,dt=2 \int_0^1 \frac{\sqrt{2d-1}}{2\pi t^{1/2}(1+t)}\,dt=\frac{\sqrt{2d-1}}{2}=\mathcal{O}(d^{1/2}).
\end{equation*}
It remains to consider the case $0 < r < 1$. As the first asymptotic value of $E$ we compute
\begin{equation}
\label{eq: Ea1}
E_1=2\int_0^1 f_a(t)\,dt.
\end{equation}
However, this formula is still not explicit since we need to take square-root of $f_a$. Next we will offer another explicit approximation. To this end, we will further simplify $f_a$ asymptotically. Because 
\[
\left(2(1-2d)(r-1)t(t+1)+\sqrt{\pi}r(t(8d+t-6)+1)\sqrt{(d-1)t}\right)\sim\sqrt{\pi}rt8d\sqrt{dt}
\]
 and 
\[
\left((r-1)(t+1)-2\sqrt{\pi}r\sqrt{(d-1)t}\right)^{2}\sim4\pi r^{2}dt
\]
we obtain 
\begin{align*}
f_a^{2} & =\frac{(1-r)\left(2(1-2d)(r-1)t(t+1)+\sqrt{\pi}r(t(8d+t-6)+1)\sqrt{(d-1)t}\right)}{8\pi^{2}t^{2}(t+1)\left((r-1)(t+1)-2\sqrt{\pi}r\sqrt{(d-1)t}\right)^{2}}\\
 & \sim\frac{(1-r)\sqrt{\pi}rt8d\sqrt{dt}}{8\pi^{2}t^{2}(t+1)4\pi r^{2}dt}=\frac{\sqrt{d}(1-r)}{4\pi^{5/2}rt^{3/2}(t+1)},
\end{align*}
which implies that 
\begin{equation*}
f_{a}\sim\frac{d^{1/4}(1-r)^{1/2}}{2\pi^{5/4}r^{1/2}t^{3/4}(t+1)^{1/2}}.\\
\end{equation*}
Hence, we obtain another approximation for $E(r,d)$ as follows.
\begin{align}
E(r,d) \sim E_2&:=\int_{0}^{1}\frac{d^{1/4}(1-r)^{1/2}}{2\pi^{5/4}r^{1/2}t^{3/4}(t+1)^{1/2}}dt\nonumber\\
 & =\frac{d^{1/4}(1-r)^{1/2}}{2\pi^{5/4}r^{1/2}}\int_{0}^{1}\frac{1}{t^{3/4}(t+1)^{1/2}}dt \nonumber\\
 & =\frac{d^{1/4}(1-r)^{1/2}}{2\pi^{5/4}r^{1/2}}\frac{8\Gamma\left(\frac{5}{4}\right)^{2}}{\sqrt{\pi}}.
 \label{eq: Ea2}
\end{align}
\end{proof}
The formal computations clearly shows that the correlation $r$ between the coefficients $\{\beta\}$ significantly influences the expected number of equilibria $E(r,d)$: 
\begin{equation*}
E(r,d)=\begin{cases}
\mathcal{O}(d^{1/2}), \quad \text{if}~~ r=0,\\
\mathcal{O}(d^{1/4}), \quad \text{if}~~ 0<r<1,\\
0, \hspace*{1.4cm} \text{if}~~ r=1.
\end{cases}
\end{equation*}
In Section~\ref{sec: numerics asymptotic} we will provide numerical verification for our formal computations.
\begin{corollary}
\label{lemma:asymSE}
The expected number of stable equilibrium points SE(r,d) follows the asymptotic behaviour 
\begin{equation*}
SE(r,d)=\begin{cases}
\mathcal{O}(d^{1/2}), \quad \text{if}~~ r=0,\\
\mathcal{O}(d^{1/4}), \quad \text{if}~~ 0<r<1,\\
0, \hspace*{1.4cm} \text{if}~~ r=1.
\end{cases}
\end{equation*}
\end{corollary}
\begin{proof}
This is a direct consequence of Theorems \ref{theorem:monotonf} and  \ref{theo: asymp}.  \end{proof}
\begin{remark} In Appendix \ref{sec: f1}, we show the following asymptotic formula for $f(1;r,d)$
\[
f(1;r,d)\sim \frac{(d-1)^{1/4}(1-r)^{1/2}}{2\sqrt{2}\pi^{5/4}r^{1/2}}.
\] 
It is worth noticing that this asymptotic behaviour is of the same form as that of $E(r,d)$.
\end{remark}
\subsection{Asymptotic behaviour of $E(r,d)$: numerical investigation}
\label{sec: numerics asymptotic}
In this section, we numerically validate the asymptotic behaviour of $E(r,d)$ for large $d$ that is obtained in the previous section using formal analytical computations. In Figure~\ref{fig:asymp E}, Table~\ref{Htable1} and Table~\ref{Htable2} we plot the ratios of the asymptotically approximations of $E(r,d)$ obtained in Section \ref{sec: assymptotic} with itself, i.e, $E_1/E(r,d)$ and $E_2/E(r,d)$, for different values of $r$ and $d$. We observe that: for $r=0$ the approximation is good; while for $0<r<1$: $E_{1}$ (respectively, $E_2$) approximates $E(r,d)$
better when $r$ is small (respectively, when $r$ is close to 1).
\begin{figure}[H]
\centering
\includegraphics[width = \linewidth]{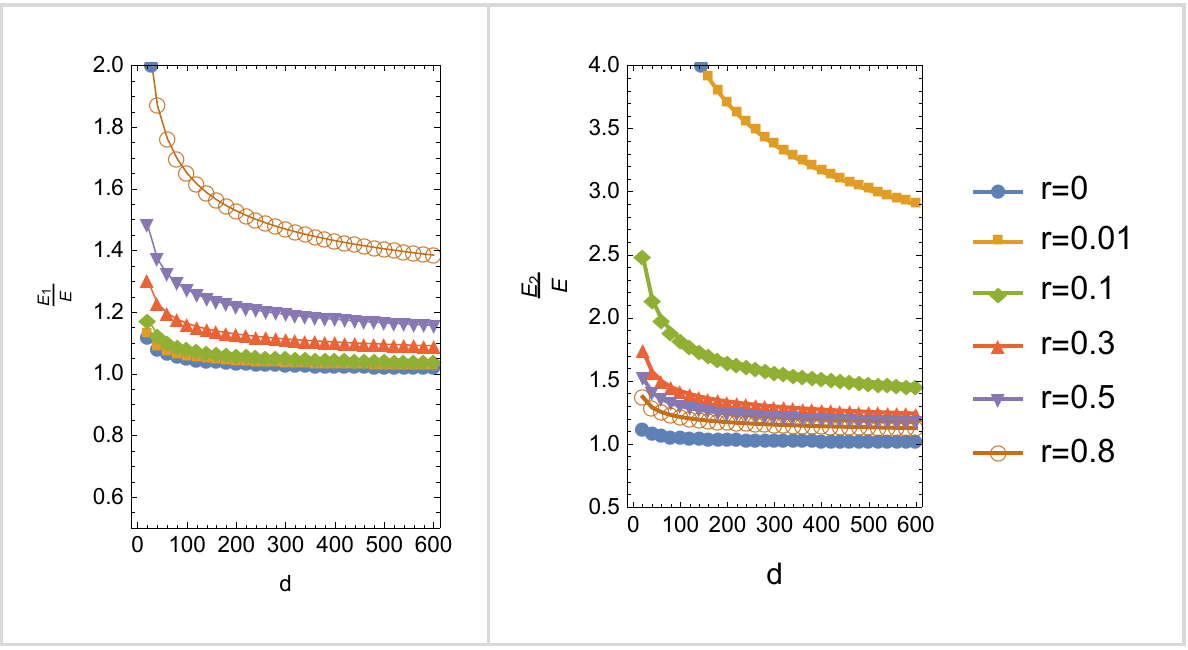}
\caption{Plot of $E_1/E(r,d)$ (\textbf{left}), and $E_2/E(r,d)$ (\textbf{right}). The figure shows that these ratios all converge to $1$ when $d$ becomes large. We also notice that $E_{2}$ approximates
$E$ better when $r$ is close to 1 while $E_{1}$ approximates $E$ better when $r$ is small. }
\label{fig:asymp E}
\end{figure} 
\begin{table}[ht]
\centering
\begin{tabular}{|c|c|c|c|c|c|c|}
\hline 
\backslashbox{$d$}{$r$}  & 0  & 0.01  & 0.1  & 0.3  & 0.5 & 0.8\tabularnewline
\hline 
20 & 0.119  & 0.126  & 0.178  & 0.305  & 0.484  & 1.106 \tabularnewline
40 & 0.08  & 0.086  & 0.128  & 0.23  & 0.373  & 0.871 \tabularnewline
120 & 0.045  & 0.049  & 0.08  & 0.154  & 0.257  & 0.616 \tabularnewline
200 & 0.034  & 0.038  & 0.065  & 0.129  & 0.219  & 0.529 \tabularnewline
320 & 0.027  & 0.03  & 0.055  & 0.111  & 0.19  & 0.461 \tabularnewline
440 & 0.023  & 0.026  & 0.049  & 0.1  & 0.172  & 0.421 \tabularnewline
600 & 0.019  & 0.023  & 0.044  & 0.091  & 0.157  & 0.385 \tabularnewline
\hline 
\end{tabular}

\caption{$\Big|\frac{E_{1}}{E}-1\Big|$\hspace{1.5cm}}
\label{Htable1}
\end{table}

\begin{table}[ht]
\centering
\begin{tabular}{|c|c|c|c|c|c|c|}
\hline 
\backslashbox{$d$}{$r$}  & 0  & 0.01  & 0.1  & 0.3  & 0.5 & 0.8\tabularnewline
\hline 
20 & 0.119  & 5.855  & 1.495  & 0.745  & 0.528  & 0.374 \tabularnewline
40 & 0.08  & 4.587  & 1.148  & 0.575  & 0.409  & 0.29 \tabularnewline
120 & 0.045  & 3.186  & 0.782  & 0.397  & 0.285  & 0.203 \tabularnewline
200 & 0.034  & 2.701  & 0.661  & 0.338  & 0.244  & 0.174 \tabularnewline
320 & 0.027  & 2.322  & 0.568  & 0.293  & 0.212  & 0.152 \tabularnewline
440 & 0.023  & 2.097  & 0.514  & 0.266  & 0.193  & 0.138 \tabularnewline
600 & 0.019  & 1.9  & 0.467  & 0.243  & 0.176  & 0.127 \tabularnewline
\hline 
\end{tabular}

\caption{$\Big|\frac{E_{2}}{E}-1\Big|$\hspace{1.5cm}}
\label{Htable2}
\end{table}
\section{Conclusion}
\label{sec: discussion}
In this paper, we have studied the mean value, $\mathit{E(r,d)}$, of the number of internal equilibria in $d$-player two-strategy random evolutionary games where the entries of the payoff matrix are correlated random variables ($r$ is the correlation). We have provided analytical formulas for $\mathit{E(r,d)}$ and  proved that it is decreasing as a function of $r$. That is, our analysis has shown  that decreasing the correlation among payoff entries leads to larger expected numbers of (stable) equilibrium points. This suggests that when payoffs obtained by a strategy  for different group compositions are less correlated, it would	 lead to higher levels of strategic or behavioural diversity in a population. Thus, one might expect that when strategies behave  conditionally on or even randomly for different  group compositions, diversity would be promoted.  Furthermore, we have shown that  the  asymptotic behaviour of $\mathit{E(r,d)}$  (and thus also of the  mean number of stable equilibrium points, $\mathit{SE(r,d)}$), i.e. when the group size $d$ is sufficiently large, is highly sensitive to the correlation value $r$. Namely,  $\mathit{E(r,d)}$ (and $\mathit{SE(r,d)}$) asymptotically behave in the order  of $d^{1/2}$ for $r  = 0$ (i.e. the payoffs  are independent for different group compositions), of $d^{1/4}$ for $0 < r  < 1$ (i.e. non-extreme correlation), and $0$  when $r  = 1$ (i.e. the payoffs  are perfectly linear).  It is also noteworthy that our numerical results showed that $\mathit{E(r,d)}$ increases with the group size $d$. In general, our findings might have important implications for the understanding of social and biological systems    given the important roles of social and biological diversities, e.g. in the evolution of cooperative behaviour and population fitness distribution \cite{Levin2000,santos2012role,pena2012group}. 

Moreover, we have explored further connections between EGT and random polynomial theory initiated in our previous works \cite{DH15,DuongHanJMB2016}. The random polynomial $P$ obtained from EGT (cf. \eqref{eq: eqn for y}) differs from three well-known classes of random polynomials, namely Kac polynomials, elliptic polynomials and Weyl polynomials, that are investigated intensively in the literature. We elaborate further this difference in  Section \ref{sec: elliptic polynomial}. In addition, as will be explained in Section \ref{sec: expected zeros Bernstein}, the set of positive roots of $P$ is the same as that of a Bernstein random polynomial. As a result, our work provides an analytical formula and asymptotic behaviour for the expected number of Bernstein random polynomials proving \cite[Conjecture 4.7]{Emiris:2010}. Thus, our work also  contributes to the literature of random polynomial theory and to further its existing connection  to EGT.

Although the expected number of internal equilibria provides macroscopic (average) information, to gain deeper insights into a multiplayer game such as possibilities of different states of biodiversity or the maintenance of biodiversity, it is crucial to analyse the probability distribution of the number of (stable) internal equilibria \cite{Levin2000,gokhale:2010pn,Stewart2016}. Thus a more subtle questions is:  what is the probability, $p_m$, with $0\leq m\leq d-1$, that  a $d$-player two-strategy game attains $m$ internal equilibria? This question has been addressed for games with a small number of players \cite{gokhale:2010pn,HTG12}. We will tackle this more intricate question for arbitary $d$ in a seperate paper~\cite{DuongTranHan2017b}. We expect that our work in this paper as well as in \cite{DuongTranHan2017b} will open up  a new exciting avenue of research in the study of
equilibrium properties of random evolutionary games. We  discuss below some  directions for future research. 

\textit{Other types of correlations.} In this paper we have assumed that the correlations $\mathit{corr}(\beta_i,\beta_j)$ are constants for all pairs $i\neq j$. This is a fairly simple relation. Generally $\mathrm{corr}(\beta_i,\beta_j)$ may depend on $i$ and $j$ as showing in Lemma \ref{lemma:relation_betaAB}. Two interesting cases that are commonly studied in interacting particle systems are: (a) exponentially decay correlations, $\mathrm{corr}(\beta_i,\beta_j)=\rho^{|i-j|}$ for some $0<\rho<1$, and (b) algebraically decay correlations, $\mathrm{corr}(\beta_i,\beta_j)=(1+|i-j|)^{-\alpha}$ for some $\alpha>0$. These types of correlations have been studied in the literature for different types of random polynomials \cite{Sambandham78, BS86, FN10}.

\textit{Universality phenomena.} Recently in \cite{TaoVu14} the authors proved, for other classes of random polynomials (such as Kac polynomials, Weyl polynomials and elliptic polynomials, see Section \ref{sec: elliptic polynomial}), an intriguing universal phenomena: the asymptotic behavior of the expected number of zeros in the non-gaussian case match that of the gaussian case once one has performed appropriate normalizations. Further research is demanded to see whether this universality phenomena holds true for the random polynomial \eqref{eq: P1}.

\section{Appendix: detailed proofs and computations} 
\label{sec: App}
This appendix consists of detailed proofs and computations of some lemmas and theorems in the main text.
\subsection{Proof of Lemma \ref{lemma:relation_betaAB}}
\label{sec: proof lemma}
We have 
\begin{align*}
\cov(\beta_i, \beta_j) &= \cov(a_i-b_i, a_j - b_j) \\ 
&=  \cov(a_i, a_j) + \cov(b_i, b_j) - \cov(a_i, b_j) - \cov(b_i, a_j) \\
&=  r_a \eta^2 +r_b \eta^2 -2 r_{ab} \eta^2 \\
&=  (r_a +r_b  -2 r_{ab}) \eta^2.
\end{align*}
Similarly, 
\[
\var(\beta_i)=\var(a_i-b_i) = \cov(a_i-b_i,a_i-b_i) = 2\eta^2 - 2 r^\prime_{ab} \eta^2 = 2 (1-r^\prime_{ab}) \eta^2.
\]
Hence, the correlation between $\beta_i$ and  $\beta_j$ is 
 \begin{equation*}
\cor(\beta_i, \beta_j) =  \frac{\cov(\beta_i, \beta_j)}{\sqrt{\var(\beta_i) \var(\beta_j) }} = \frac{(r_a + r_b - 2r_{ab})\eta^2}{2(1-r^\prime_{ab})\eta^2} = \frac{r_a + r_b - 2r_{ab}}{2(1-r^\prime_{ab})}.
\end{equation*}
\subsection{Proof of Theorem \ref{theo: integral form}--(1)}
\label{sec: symmetry of game}
We prove \eqref{eq: f(1/t) vs f(t)}.  We recall the following notations that have been used in the proof of Theorem \ref{theorem:monotonf}.
\begin{align*}
&M_1=M_1(t,d)=\sum_{i=0}^{d-1}\begin{pmatrix}
d-1\\i
\end{pmatrix}^2t^{2i}, \quad M_2=M_2(t,d)=(1+t)^{2(d-1)},
\\& A_1=A_1(t,d)=\sum\limits_{i=0}^{d-1}i^2\begin{pmatrix}
d-1\\
i
\end{pmatrix}^2t^{2(i-1)}, \quad A_2=A_2(t,d)=(d-1)^2(1+t)^{2(d-2)}
\\& B_1=B_1(t,d)=\sum\limits_{i=0}^{d-1}i\begin{pmatrix}
d-1\\
i
\end{pmatrix}^2t^{2i-1}, \quad B_2=B_2(t,d)=(d-1)(1+t)^{2d-3},
\\& M=M(t;r,d)=(1-r)M_1+ r M_2, \quad A=A(t;r,d)=(1-r) A_1+r A_2, 
\\& B=B(t;r,d)=(1-r)B_1+ r B_2.
\end{align*}
Then the density function $f(t;r,d)$ is expressed in terms of $M, A$ and $B$ as (for simplicity of notation we drop $r, d$ in $f$ in the following)
\begin{equation}
\label{eq: f interms of ABM}
f(t)=\frac{1}{\pi}\frac{\sqrt{AM-B^2}}{M}.
\end{equation}
Next we compute $f(1/t)$. According to \cite{DuongHanJMB2016}, we have the following relations, where $'$ denotes a derivative with respect to $t$,
\begin{align}
A_1(t)&=\frac{1}{4t}(tM_1'(t))'=\frac{1}{4t}(M_1'(t)+tM_1''(t)),\quad B_1(t)=\frac{1}{2}M_1'(t),\quad M_1(1/t)=t^{2-2d}M_1(t)\label{eq: abm1}
\\A_1(1/t)&=\frac{t}{4}\Big[M_1'(1/t)+\frac{1}{t}M_1''(1/t)\Big]=\frac{1}{4}t^{4-2d}\left[4(d-1)^2M_1(t)+(5-4d)tM_1'(t)+t^2M_1''(t)\right],\notag
\\ B_1(1/t)&=\frac{1}{2}M_1'(1/t)=-t^{3-2d}\left[(1-d)M_1(t)+\frac{1}{2}tM_1'(t)\right].\notag
\end{align}
Using the relations between $A_1, B_1$ and $M_1$ in \eqref{eq: abm1}, we transform further $A_1(1/t)$ and $B_1(1/t)$
\begin{align*}
A_1(1/t)&=\frac{1}{4}t^{4-2d}\Big[4(d-1)^2M_1(t)+4(1-d)tM_1'(t)+t(M_1'(t)+tM_1''(t))\Big]
\\&=t^{4-2d}\Big[4(d-1)^2M_1(t)+4(1-d)tM_1'(t)+t^2 A_1(t)\Big],
\\ B_1(1/t)&=-t^{3-2d}\left[(1-d)M_1(t)+\frac{1}{2}tM_1'(t)\right]=t^{2-2d}\Big[(d-1)M_1(t)-tB_1(t)\Big].
\end{align*}
Using explicit formulas of $M_2, A_2$ and $B_2$, we get
\begin{equation}
M_2(1/t)=t^{2-2d}M_2(t), \quad A_2(1/t)=t^{4-2d} A_2(t), \quad B_2(1/t)=t^{3-2d}B_2(t).
\end{equation}
Therefore, we obtain
\begin{align*}
&M(1/t)=(1-r)M_1(1/t)+ rM_2(1/t)=t^{2-2d}[(1-r)M_1(t)+rM_2(t)]=t^{2-2d}M(t),
\\&A(1/t)=t^{4-4d}\left[(1-r)\Big((d-1)^2 M_1(t)+(1-d)tM_1'(t)+t^2 A_1(t)\Big)+r A_2(t)\right],
\\&B(1/t)=t^{3-2d}\left[(1-r)\Big((d-1)M_1(t)-tB_1(t)\Big)+rB_2(t)\right],
\\&M(1/t)A(1/t)=t^{6-4d}\Bigg[(1-r)^2\Big((d-1)^2 M_1(t)+(1-d)tM_1'(t)+t^2 A_1(t)\Big)M_1(t)
\\&\hspace*{3.5cm}+r(1-r)\bigg(\Big((d-1)^2 M_1(t)+(1-d)tM_1'(t)+t^2 A_1(t)\Big)M_2(t)+A_2(t)M_1(t)\bigg)
\\&\hspace*{4cm}+r^2A_2(t)M_2(t)\Bigg],
\\& B(1/t)^2=t^{6-4d}\Bigg[(1-r)^2\Big((d-1)M_1(t)-tB_1(t)\Big)^2+2r(1-r)\Big((d-1)M_1(t)-tB_1(t)\Big) B_2(t)+r^2B_2(t)^2\Bigg].
\end{align*}
So we have
\begin{align}
\label{eq: AM-B2}
&M(1/t)A(1/t)-B(1/t)^2\notag
\\&=t^{6-4d}\Bigg[(1-r)^2\Big((1-d)tM_1(t)M_1'(t)+t^2A_1(t)M_1(t)+2(d-1)M_1(t)B_1(t)-t^2B_1(t)^2\Big)\notag
\\&\qquad\qquad+ r(1-r)\bigg((d-1)^2M_1(t)M_2(t)+(1-d)tM_1'(t)M_2(t)+t^2A_1(t)M_2(t)+A_2(t)M_1(t)\notag
\\&\qquad\qquad\qquad-2\Big((d-1)M_1(t)-tB_1(t)\Big)B_2(t)\bigg)+ r^2(A_2(t)M_2(t)-B_2(t)^2)\Bigg].
\end{align}
Using the relations \eqref{eq: abm1} and explicit formulas of $A_2, B_2, M_2$ we get
\begin{align*}
&A_2(t)M_2(t)-B_2^2(t)=0,
\\&(1-d)tM_1(t)M_1'(t)+2(d-1)M_1(t)B_1(t)=(d-1)M_1(t)\Big[2B_1(t)-M_1'(t)\Big]=0,
\\&(d-1)^2M_1(t)M_2(t)+A_2(t)M_1(t)-2(d-1)M_1(t)B_2(t)
\\&\qquad=M_1(t)\Big((d-1)M_2(t)+A_2(t)-2(d-1)B_2(t)\Big)
\\&\qquad=t^2 M_1(t)A_2(t),
\\& (1-d)tM_1'(t)M_2(t)+2tB_1(t)B_2(t)=2(1-d)tB_1(t)M_2(t)+2tB_1(t)B_2(t)
\\&\qquad=B_1(t)\Big(2(1-d)tM_2(t)+2tB_2(t)\Big)
\\&\qquad=-2t^2B_1(t)B_2(t).
\end{align*}
Substituting these computations into \eqref{eq: AM-B2}, we obtain
\begin{align*}
&M(1/t)A(1/t)-B(1/t)^2
\\&\qquad=t^{8-4d}\Bigg[(1-r)^2\Big(A_1(t)M_1(t)-B_1(t)^2\Big)+r(1-r)\Big(M_1(t)A_2(t)+M_2(t)A_1(t)-2B_1(t)B_2(t)\Big)\Bigg]
\\&\qquad=t^{8-4d}\Bigg[\Big((1-r)M_1(t)+rM_2(t)\Big)\Big((1-r)A_1(t)+rA_2(t)\Big)-\Big((1-r)B_1(t)+rB_2(t)\Big)^2\Bigg]
\\&\qquad=t^{8-4d}\Big(M(t)A(t)-B(t)^2\Big).
\end{align*}
Finally, we get
\begin{equation*}
f(1/t)=\frac{1}{\pi}\frac{\sqrt{A(1/t)M(1/t)-B(1/t)^2}}{M(1/t)}=\frac{1}{\pi}\frac{t^{4-2d}\sqrt{A(t)M(t)-B^2(t)}}{t^{2-2d}M(t)}=\frac{1}{\pi}t^2\frac{\sqrt{A(t)M(t)-B^2(t)}}{M(t)}=t^2f(t).
\end{equation*}
\subsection{Proof of Lemma~\ref{lem: symmetry of betas}}
\label{sec: symmetry of betas}
The probability distribution, $f_Z$, of $Z=X-Y$ can be found via the joint probability distribution $f_{X,Y}$ as 
\[
f_{Z}(z)=\int_{-\infty}^{\infty} f_{X,Y}(x,x-z)\,dx=\int_{-\infty}^{\infty} f_{X,Y}(y+z,y)\,dy. 
\]
Therefore, using the symmetry of $f_{X,Y}$ we get
\[
f_Z(-z)=\int_{-\infty}^{\infty} f_{X,Y}(x,x+z)\,dx=\int_{-\infty}^{\infty} f_{X,Y}(x+z,x)\,dx=f_Z(z). 
\]
If $X$ and $Y$ are iid with the common probability distribution $f$ then
\[
f_{X,Y}(x,y)=f(x)f(y),
\]
which is symmetric with respect to $x$ and $y$, i.e., $X$ and $Y$ are exchangeable.
\subsection{Computations of $f(1;r,d)$}
\label{sec: f1}
Substituting $t=1$ into expressions of $A, B, M$ at the beginning of the proof of Theorem \ref{theo: integral form}, we obtain
\begin{align*}
&M(1;r,d)=(1-r)\sum\limits_{k=0}^{d-1}\begin{pmatrix}
d-1\\
k
\end{pmatrix}^2+r\, 2^{2(d-1)}=(1-r)\begin{pmatrix}
2(d-1)\\
d-1
\end{pmatrix}+r\, 2^{2(d-1)},
\\& A(1;r,d)=(1-r)(d-1)^2M(1;r,d-1)+r(d-1)^22^{2(d-2)}=(1-r)(d-1)^2\begin{pmatrix}
2(d-2)\\
d-2
\end{pmatrix}+r(d-1)^22^{2(d-2)},
\\& B(1;r,d)=(1-r)\sum\limits_{k=1}^{d-1}k\begin{pmatrix}
d-1\\
k
\end{pmatrix}^2+r(d-1)2^{2d-3)}=(1-r)\frac{d-1}{2}\begin{pmatrix}
2(d-1)\\
d-1
\end{pmatrix}+r(d-1)2^{2d-3}.
\end{align*}
Therefore,
\begin{align*}
AM-B^2&=(1-r)^2(d-1)^2\begin{pmatrix}
2(d-1)\\
d-1
\end{pmatrix}\left[\begin{pmatrix}
2(d-2)\\
d-2
\end{pmatrix}-\frac{1}{4}\begin{pmatrix}
2(d-1)\\
d-1
\end{pmatrix}\right]
\\&\qquad+r(1-r)(d-1)^22^{2(d-2)}\left[4\begin{pmatrix}
2(d-2)\\
d-2
\end{pmatrix}+\begin{pmatrix}
2(d-1)\\
d-1
\end{pmatrix}-2\begin{pmatrix}
2(d-1)\\
d-1
\end{pmatrix}\right]
\\&=(1-r)^2(d-1)^2\begin{pmatrix}
2(d-1)\\
d-1
\end{pmatrix}\left[\begin{pmatrix}
2(d-2)\\
d-2
\end{pmatrix}-\frac{1}{4}\begin{pmatrix}
2(d-1)\\
d-1
\end{pmatrix}\right]
\\&\qquad+r(1-r)(d-1)^22^{2(d-1)}\left[\begin{pmatrix}
2(d-2)\\
d-2
\end{pmatrix}-\frac{1}{4}\begin{pmatrix}
2(d-1)\\
d-1
\end{pmatrix}\right]
\\&=(1-r)(d-1)^2\left[\begin{pmatrix}
2(d-2)\\
d-2
\end{pmatrix}-\frac{1}{4}\begin{pmatrix}
2(d-1)\\
d-1
\end{pmatrix}\right]\left[(1-r)\begin{pmatrix}
2(d-1)\\
d-1
\end{pmatrix}+r 2^{2d-1}\right].
\end{align*}
Substituting this expression and that of $M$ into \eqref{eq: f interms of ABM}, we get
\begin{align*}
f(1;r,d)&=\frac{1}{\pi}\frac{\sqrt{AM-B^2}}{M}
\\&=\frac{1}{\pi}(d-1)\sqrt{1-r}\times\sqrt{\frac{\begin{pmatrix}
2(d-2)\\
d-2
\end{pmatrix}-\frac{1}{4}\begin{pmatrix}
2(d-1)\\
d-1
\end{pmatrix}}{(1-r)\begin{pmatrix}
2(d-1)\\
d-1
\end{pmatrix}+r\, 2^{2(d-1)}}}
\\&=\frac{1}{\pi}(d-1)\sqrt{1-r}\times\sqrt{\frac{\begin{pmatrix}
2(d-1)\\
d-1
\end{pmatrix}\frac{1}{4(2d-3)}}{(1-r)\begin{pmatrix}
2(d-1)\\
d-1
\end{pmatrix}+r\, 2^{2(d-1)}}}
\\&=\frac{1}{\pi}\frac{d-1}{2\sqrt{2d-3}}\sqrt{\frac{(1-r)\begin{pmatrix}
2(d-1)\\
d-1
\end{pmatrix}}{(1-r)\begin{pmatrix}
2(d-1)\\
d-1
\end{pmatrix}+r\, 2^{2(d-1)}}}.
\end{align*}
If $r=1$ then $f(1;r,d)=0$. If $r<1$ then
\begin{align*}
f(1;r,d)=\frac{1}{\pi}\frac{d-1}{2\sqrt{2d-3}}\sqrt{\frac{1}{1+\alpha}}\quad\text{where}\quad
\alpha=\frac{r}{1-r}\frac{2^{2(d-1)}}{\begin{pmatrix}
2(d-1)\\
d-1
\end{pmatrix}}.
\end{align*}
By Stirling formula, we have
\[
\begin{pmatrix}
2n\\n
\end{pmatrix}\sim \frac{4^n}{\sqrt{\pi n}} \quad\text{for large $n$}.
\]
It implies that for $0< r<1$ and for large $d$
\[
\alpha\sim\frac{r}{1-r}\sqrt{\pi (d-1)}\quad \text{and}\quad f(1;r,d)\sim \frac{1}{\pi}\frac{d-1}{2\sqrt{2d-3}}\sqrt{\frac{1}{1+\frac{r}{1-r}\sqrt{\pi (d-1)}}}\sim \frac{(d-1)^{1/4}(1-r)^{1/2}}{2\sqrt{2}\pi^{5/4}r^{1/2}}.
\]
\subsection{Some technical lemmas used in proof of Theorem \ref{theorem:monotonf}}
We need the following proposition.
\begin{proposition} The following inequality holds
\label{lemma:support_deriv_theorem1}
\begin{equation} 
2(t+1)B\Big[B_1(1+t)-M_1(d-1)\Big]<M\Big[A_1(1+t)^2-M_1(d-1)^2\Big].
\end{equation}
\end{proposition}

\label{appendix:proof_support_deriv_theorem1}
To prove Proposition \ref{lemma:support_deriv_theorem1}, we need several auxiliary lemmas. We note that throughout this section
\[
x=\frac{1+t^2}{1-t^2}, \quad 0<t<1,
\]
and $P_d(z)$ is the Legendre polynomial of degree $d$ which is defined through the following recurrent relation
\begin{equation}
\label{eq: Pd relation}
(2d+1) z P_d(z)=(d+1) P_{d+1}(z)+d\, P_{d-1}(z); \quad P_0(z)=1, \quad P_1(z)=z.
\end{equation}
We refer to \cite{DuongHanJMB2016} for more information on the Legendre polynomial and its connections to evolutionary game theory.

\begin{lemma} It holds that
\label{lemma:lim_Consec_Legendre}
\begin{equation*}
\lim_{d\to \infty}\frac{P_d(x)}{P_{d+1}(x)}=x-\sqrt{x^2-1}.
\end{equation*}
Note that $x=\frac{1+t^2}{1-t^2}$, we can write the above limit as
\begin{equation}
\label{eq: limit of Pd}
\lim_{d\to \infty}\frac{P_d(x)}{P_{d+1}(x)}=\frac{1-t}{1+t}.
\end{equation}
\end{lemma}
\begin{proof}
According to \cite[Lemma 4]{DuongHanJMB2016} we have 
\begin{equation*}
P_d(x)^2\leq P_{d+1}(x) P_{d-1}(x).
\end{equation*}
Since $P_d(x)>0$, we get
\begin{equation}
\label{eq: Pd/Pd}
x\geq \frac{1}{x} = \frac{P_{0}(x)}{P_{1}(x)}\geq \frac{P_1(x)}{P_2(x)}  \geq \ldots \geq  \frac{P_{d-1}(x)}{P_{d}(x)} \geq   \frac{P_{d}(x)}{P_{d+1}(x)}\geq 0.
\end{equation}
Therefore, there exists a function $0\leq f(x)\leq \frac{1}{x}$ such that
\[
\lim_{d\to \infty}\frac{P_d(x)}{P_{d+1}(x)}=f(x).
\]
From the recursive relation \eqref{eq: Pd relation} we have
\[
(2d+1)x=(d+1)\frac{P_{d+1}(x)}{P_d(x)}+d\frac{P_{d-1}(x)}{P_d(x)},
\]
which implies that
\[
\frac{d+1}{d}=\frac{\frac{P_{d-1}(x)}{P_d(x)}-x}{x-\frac{P_{d+1}(x)}{P_d(x)}}.
\]
Taking the limit $d\to\infty$ both sides we obtain
\[
1=\frac{f(x)-x}{x-\frac{1}{f(x)}}.
\]
Solving this equation for $f(x)$, requiring that $0\leq f(x)\leq\frac{1}{x}\leq x$ we obtain $f(x)=x-\sqrt{x^2-1}$.
\end{proof}

\begin{lemma} 
\label{lemma:legendreIneq1}
The following inequalities hold 
\begin{equation}
(1-t)^2 \leq (1-t^2)  \frac{P_{d}(x)}{P_{d+1}(x)}  \leq 1+t^2.
\end{equation} 
\end{lemma} 
\begin{proof} By dividing by $1-t^2$, the required inequalities are equivalent to  (recalling that $0<t<1$)
\begin{equation*}
\frac{1-t}{1+t}  \leq  \frac{P_{d}(x)}{P_{d+1}(x)} \leq x,
\end{equation*} 
which are true following from \eqref{eq: limit of Pd} and \eqref{eq: Pd/Pd}.
\end{proof}
\begin{lemma} The following equality  holds 
\label{proposition:common_factor_deriv}
\begin{equation} 
2(d-1)t \left[B_1(1+t) - M_1 (d-1) \right] = (t-1)   \left[A_1 (1+t)^2 - M_1 (d-1)^2\right]. 
\end{equation} 
\end{lemma}
\begin{proof}  The stated equality is simplified to 
 \begin{equation} 
 \label{eq:commonfactor1}
A_1 (t^2 - 1) + M_1 (d-1)^2 - 2(d-1) t \, B_1 = 0.
\end{equation} 
We use the following results from \cite[Lemma 3 \& Section 6.2]{DuongHanJMB2016}
\begin{align}
A_1(t,d)&= (d-1)^2 M_1(t,d-1) = (d-1)^2 (1-t^2)^{d-2}P_{d-2}(x), \quad M_1 = (1-t^2)^{d-1}P_{d-1}(x),\label{eq: A1 and M1}
\\ B_1 &= \frac{M_1^\prime}{2}\label{eq: B1 and M1}
\\& = M_1 \left(\frac{-t\,(d-1)}{1-t^2}+\frac{2t}{(1-t^2)^2}\frac{P_{d-1}'}{P_{d-1}}\left(x\right)\right) \nonumber
\\&= M_1 \left(\frac{-t\,(d-1)}{1-t^2}+\frac{2t}{(1-t^2)^2} \frac{(d-1)(1-t^2)^2}{4t^2}\left(\frac{1+t^2}{1-t^2}-\frac{P_{d-2}(x)}{P_{d-1}(x)}\right)\right) \nonumber\\
&= M_1 \left(\frac{-t\,(d-1)}{1-t^2}+ \frac{d-1}{2t}\left(\frac{1+t^2}{1-t^2}-\frac{P_{d-2}(x)}{P_{d-1}(x)}\right)\right)   \nonumber\\
&= (d-1)\left(-t(1-t^2)^{d-2} P_{d-1}(x)+ \frac{(1+t^2)(1-t^2)^{d-2} P_{d-1}(x) - (1-t^2)^{d-1}  P_{d-2}(x) }{2t}   \right).  \nonumber
\end{align}
Substituting these expressions into the left-hand side of \eqref{eq:commonfactor1} we obtain $0$ as required.
\end{proof} 

\begin{lemma} 
\label{proposition:supporting_inequalities_theorem1}
The following inequality holds 
\begin{align} 
& (t -1) \left[B_1(1+t) - M_1 (d-1) \right] \geq 0, \label{eq1}
\\& (t^2-1)B_1 - (d-1) t \, M_1 \leq 0 \label{eq2},
\\& (t^2-1)(B_2-B_1) - (d-1)t (M_2-M_1) \leq0\label{eq3}.
 \end{align} 
\end{lemma} 
\begin{proof}
We prove \eqref{eq1} first. Since $M_1 > 0$, \eqref{eq1} is simplified to 
$$ (t^2-1)\frac{B_1}{M_1} - (d-1)(t -1) \geq 0. $$
Using the relation \eqref{eq: B1 and M1} between $B_1$ and $M_1$ we obtain 
\begin{align*}
(t^2-1)\frac{B_1}{M_1} - (d-1)(t -1)&=(t^2-1)\frac{M_1^\prime}{2 M_1} - (d-1)(t -1)
\\&=(t^2-1)\Bigg[\frac{-\,t\,(d-1)}{1-t^2}+\frac{2t}{(1-t^2)^2}\frac{P_{d-1}'}{P_{d-1}}\left(\frac{1+t^2}{1-t^2}\right)\Bigg]-(d-1)(t -1) 
\\&=(d-1)  + \frac{2t}{t^2-1}\frac{P_{d-1}'}{P_{d-1}}\left(x \right).
\end{align*}
Now using the following relation \cite[Eq. (49)]{DuongHanJMB2016}
\begin{equation}
\label{eq: P'd/Pd}
\frac{P_{d-1}'(x)}{P_{d-1}(x)}=\frac{d-1}{x^2-1}\left(x-\frac{P_{d-2}(x)}{P_{d-1}(x)}\right) =\frac{(d-1)(1-t^2)^2}{4t^2}\left(\frac{1+t^2}{1-t^2}-\frac{P_{d-2}(x)}{P_{d-1}(x)}\right),
\end{equation}
we obtain
\begin{align*}
(t^2-1)\frac{B_1}{M_1} - (d-1)(t -1)&=(d-1)\left(1-\frac{1+t^2}{2t}  - \frac{t^2-1}{2t}\frac{P_{d-2}(x)}{P_{d-1}(x)} \right)
\\&=-\frac{d-1}{2t}\Big[(1-t)^2-(1-t^2)\frac{P_{d-1}}{P_{d-1}}(x)\Big]
\\&\geq 0,
\end{align*}
where the last inequality follows from Lemma \ref{lemma:legendreIneq1}. This establishes \eqref{eq1}.
\\ \ \\
Next we prove \eqref{eq2}, which can be simplified to

$$ (d-1)\left(-\frac{1+t^2}{2t}  - \frac{t^2-1}{2t}\frac{P_{d-2}(x)}{P_{d-1}(x)} \right) \leq 0,  $$
which is in turn equivalent to 
$$ (1-t^2)\frac{P_{d-2}(x)}{P_{d-1}(x)}  \leq 1 + t^2. $$
This has been proved in Lemma \ref{lemma:legendreIneq1}. \\ \\
Finally we prove \eqref{eq3}. First we simplify 
$$ (t^2-1)B_2- (d-1)t \,M_2 = (d-1)(t^2-1)(1+t)^{2d-3} - (d-1) t (1+t)^{2d-2} = - (d-1) (1+t)^{2d-2}.$$
Thus \eqref{eq3} is equivalent to 
$$ (d-1)\left(\frac{1+t^2}{2t}  + \frac{t^2-1}{2t}\frac{P_{d-2}(x)}{P_{d-1}(x)} - (1+t)^{2d-2} \right) \leq 0. $$
This clearly holds because $t \geq 0$ and from the proof of the first inequality we already know that 
$$
\frac{1+t^2}{2t}  + \frac{t^2-1}{2t}\frac{P_{d-2}(x)}{P_{d-1}(x)} - 1 \leq 0.
$$
Thus we finish the proof of the lemma.
\end{proof}
We are now ready to provide a proof of Proposition \ref{lemma:support_deriv_theorem1}. 
\begin{proof}[Proof of Proposition \ref{lemma:support_deriv_theorem1} ]
From Lemma \ref{proposition:common_factor_deriv}, since $M_1, A_1, B_1$ are polynomials (of $t$) with integer coefficients, there exists a polynomial $S(t)$ such that 
  $$B_1(1+t) - M_1 (d-1) = (t-1) S(t) \quad\text{and}\quad A_1 (1+t)^2 - M_1 (d-1)^2 = 2(d-1) t \ S(t)  $$
If follows from \eqref{eq1} that $S(t)\geq 0$. Next we will prove that
\begin{align}
& 2(t+1)B_1 \left[B_1(1+t) - M_1 (d-1) \right] \leq  M_1 \left[A_1 (1+t)^2 - M_1 (d-1)^2\right], \label{eq4}
\\& 2(t+1)(B_2-B_1) \left[B_1(1+t) - M_1 (d-1) \right] \leq  (M_2 - M_1) \left[A_1 (1+t)^2 - M_1 (d-1)^2\right].\label{eq5}
\end{align}
Indeed, these inequalities can be rewritten as 
\begin{align*}
& 2S(t) \left[ (t^2-1)b_1 - (d-1) t \, m_1 \right] < 0,
\\ & 2S(t) \left[ (t^2-1)(b_2-b_1) - (d-1)t (m_2-m_1) \right] < 0,
\end{align*}
which hold due to Lemma \ref{proposition:supporting_inequalities_theorem1}. Multiplying \eqref{eq5} with $r>0$ and adding with \eqref{eq4} yields the assertion of Proposition \ref{lemma:support_deriv_theorem1}.
\end{proof}
\subsection{Comparison with known results for other classes of random polynomials}
\label{sec: elliptic polynomial}
The distribution and expected number of real zeros of a random polynomial has been a topic of intensive research dating back to 1932 with Block and P\'{o}lya \cite{BP32}, see for instance the monograph~\cite{BS86} for a nice exposition and~\cite{TaoVu14,NNV15TMP} for recent results and discussions. The most general form of a random polynomial is given by
\begin{equation}
\P_d(z)=\sum_{i=0}^{d} c_i\,\xi_i\,z^i,
\end{equation}
where $c_i$ are deterministic coefficients which may depend on both $d$ and $i$  and $\xi_i$ are random variables. The most three well-known classes of polynomials are
\begin{enumerate}[(i)]
\item Kac polynomials: $c_i\colonequals 1$,
\item Weyl (or flat) polynomials: $c_i\colonequals\frac{1}{i!}$,
\item Elliptic (or binomial) polynomials: $c_i\colonequals\sqrt{\begin{pmatrix}
d\\
i
\end{pmatrix}}$.
\end{enumerate}
The expected number of real zeros of these polynomials when $\{\xi_i\}$ are i.i.d standard normal variables are, respectively, $E_K\sim \frac{2}{\pi}\log d$, $E_{W}\sim \frac{2}{\pi}\sqrt{d}$ and $E_E=\sqrt{d}$, see e.g., \cite{TaoVu14} and references therein. Random polynomials in which $\xi_i$ are correlated random variables have also attracted  considerable attention, see e.g., \cite{Sambandham76, Sambandham77, Sambandham78, Sambandham79, BS86, FN05, FN10, FN11} and references therein. Particularly, when $\{\xi_i\}$ satisfy the same assumption as in this paper, it has been shown, in \cite{Sambandham76} for the Kac polynomial that $E_K\sim \frac{2}{\pi}\sqrt{1-r^2}\log d$, and in \cite{FN11} for elliptic polynomials that $E_E\sim \frac{\sqrt{d}}{2}$.

The random polynomial $P$ arising from evolutionary game theory  in this paper, see Equation \eqref{eq: P1}, corresponds to  $c_i=\begin{pmatrix}
d-1\\
i
\end{pmatrix}$; thus it differs from all the above three classes. In Section \ref{sec: expected zeros Bernstein} below we show that a root of $P$ is also a root of the Bernstein polynomial. Therefore we also obtain an asymptotic formula for the expected number of real zeros of the random Bernstein polynomial. We anticipate that evolutionary game theory and random polynomial theory have deeply undiscovered connections in different scenarios. We shall continue  this development in~\cite{DuongTranHan2017TMP}.
\subsection{On the expected number of  real zeros of a random Bernstein polynomial of degree $d$} 
\label{sec: expected zeros Bernstein}
Similarly as in \cite[Corollary 2]{DuongHanJMB2016}, as a by-product of Theorem \ref{theo: asymp}, we obtain an asymptotic formula for the expected number of  real zeros, $E_{\B}$, of a random Bernstein polynomial of degree $d$
\begin{equation*}
\B(x)=\sum\limits_{k=0}^{d}\beta_k\begin{pmatrix}
d\\
k
\end{pmatrix}x^k\,(1-x)^{d-k},
\end{equation*} 
where $\beta_k$ are i.i.d. standard normal distributions. Indeed, by changing of variables $y=\frac{x}{1-x}$ as in Section \ref{sec: replicator}, zeros of $\B(x)$ are the same as those of the following random polynomial
\[
\tilde{\B}(y)=\sum\limits_{k=0}^{d}\beta_k\begin{pmatrix}
d\\
k
\end{pmatrix}y^k.
\]
As a consequence of Theorem \ref{theo: asymp}, the expected number of  real zeros, $E_{\B}$, of a random Bernstein polynomial of degree $d$ is given by
\begin{equation}
E_{\B}=2E(0,d+1)\sim \sqrt{2d+1}.
\end{equation}
This proves Conjecture 4.7 in \cite{Emiris:2010}. Connections between EGT and Bernstein polynomials have also been discussed in \cite{PENA2014}.
\section*{Acknowledgements} 
This paper was written partly when M. H. Duong was at the Mathematics Institute, University of Warwick and was supported by ERC Starting Grant 335120.  M. H. Duong and T. A. Han acknowledge Research in Pairs Grant (No. 41606) by the London Mathematical Society to support their collaborative research. 
\bibliographystyle{alpha}
\newcommand{\etalchar}[1]{$^{#1}$}

\end{document}